\documentclass[a4paper,11pt]{article}
\usepackage{ae} 
\usepackage[T1]{fontenc}
\usepackage[ansinew]{inputenc}
\usepackage[leqno,namelimits]{amsmath}
\usepackage{amsthm}
\usepackage{amssymb}
\usepackage{graphicx}

\newtheorem{thm}{Theorem}
\newtheorem{cor}[thm]{Corollary}
\newtheorem{lem}[thm]{Lemma}
\newtheorem{prop}[thm]{Proposition}

\numberwithin{equation}{section}

 \DeclareMathOperator{\RE}{Re}
 \DeclareMathOperator{\IM}{Im}

 \newcommand{\M}{\mathcal{M}}
 
 \newcommand{\K}{\mathcal{K}}
 
 \newcommand{\Real}{\mathbb{R}}
 \newcommand{\Complex}{\mathbb{C}}

 \newcommand{\set}[1]{\left\{#1\right\}}
 \newcommand{\seq}[1]{\left<#1\right>}
 \newcommand{\norm}[1]{\left\Vert#1\right\Vert}

\author{ Aik. Aretaki and J. Maroulas\footnote{Department of Mathematics, National\, Technical \,University \,of
Athens, Zografou Campus, Athens 15780, Greece. E-mail address:
maroulas@math.ntua.gr.}}
\title{The higher rank numerical range\\ of matrix polynomials}
\begin{document}
\maketitle
\begin{abstract}
The notion of the higher rank numerical range $\Lambda_{k}(L(\lambda))$ for matrix polynomials
$L(\lambda)=A_{m}\lambda^{m}+\ldots+A_{1}\lambda+A_{0}$ is introduced here and
some fundamental geometrical properties are investigated. Further, the sharp points of $\Lambda_{k}(L(\lambda))$ are
defined and their relation to the nume\-rical range $w(L(\lambda))$ is presented. A connection of $\Lambda_{k}(L(\lambda))$ with
the vector-valued higher rank numerical range $\Lambda_{k}(A_{0}, \ldots, A_{m})$ is also discussed.
\end{abstract}
{\small\textbf{Key words}: higher rank numerical range, matrix polynomials, quantum error correction.}


\section{Introduction}
Let $\M_{n}(\Complex)$ be the algebra of matrices $A=[a_{ij}]_{i,j=1}^{n}$
with entries $a_{ij}\in \Complex$ and
\[
L(\lambda)=A_{m}\lambda^{m}+A_{m-1}\lambda^{m-1}+\ldots +A_{1}\lambda+A_{0}
\]
be a matrix polynomial with $A_{i}\in\M_{n}(\Complex)$ and $A_{m}\neq 0$. For a positive integer $k\geq 1$, we define the
\textit{higher rank numerical range} of $L(\lambda)$ as
\begin{equation}\label{rel1}
\Lambda_{k}(L(\lambda))=\set{\lambda\in\Complex : PL(\lambda)P=0_{n}\,\, for\,\, some\,\, P\in\mathcal{P}_{k}},
\end{equation}
where $\mathcal{P}_{k}$ is the set of all  orthogonal projections $P$ of $\Complex^{n}$ onto any $k$-dimensional
subspace $\K$ of $\Complex^{n}$. Equivalently,
\begin{equation}\label{re11}
\Lambda_{k}(L(\lambda))=\set{\lambda\in\Complex : Q^{*}L(\lambda)Q=0_{k}\,\, for\,\, some\,\, Q\in\M_{n,k}\,\,with\,\,Q^{*}Q=I_{k}},
\end{equation}
since $P=QQ^{*}$, with $Q\in\M_{n,k}(\Complex)$ and $Q^{*}Q=I_{k}$.
In case $k=1$, the set reduces to the well known \emph{numerical range} $w(L(\lambda))$ of $L(\lambda)$ \cite{Li-Rodman}
\begin{equation}\label{rel2}
\Lambda_{1}(L(\lambda))\equiv w(L(\lambda))=\set{\lambda\in\Complex : x^{*}L(\lambda)x=0\,\, for\,\, some\,\, x\in\Complex^{n}, \norm x=1}.
\end{equation}
The \eqref{rel1} (or \eqref{re11}) is an interesting generalization of numerical range, since matrix polynomials play a significant role in several
problems of computational chemistry and structural molecular biology \cite{Emiris}. They consist algebraic tools to computing all
conformations of ring molecules and  they model va\-rious problems in terms of polynomial equations.

If $L(\lambda)=I\lambda-A$, then
\begin{eqnarray}\label{rel3}
 \nonumber\Lambda_{k}(I\lambda-A) & = & \set{\lambda\in\Complex : PAP=\lambda P\,\, for\,\, some\,\, P\in\mathcal{P}_{k}}\\
    & = & \set{\lambda\in\Complex : Q^{*}AQ=\lambda I_{k}\,,\,\,Q^{*}Q=I_{k},\,\,Q\in \M_{n,k}(\Complex)},
\end{eqnarray}
namely, it coincides with the \emph{higher rank numerical range} of a matrix $A\in\M_{n}$. The concept of higher rank numerical range
has been studied extensively by Choi \textit{et al} in \cite{Choi,C-H-K-Z,C-K-Z-q,C-K-Z}
and later by other researchers in \cite{Poon-Li-Sze,Li-Sze,Hugo}. We should note that for $k=1$, $\Lambda_{k}(I\lambda-A)$ yields
the classical  \emph{numerical range} of a matrix $A$, i.e.
\begin{equation}\label{rel4}
F(A)=\set{ x^{*}Ax : x\in \Complex^{n}, \norm x=1},
\end{equation}
whose basic properties can be found in ~\cite{Rao,Halmos,H.J.T}.

A multi-dimensional higher rank numerical range is the
\emph{joint higher rank numerical range} \cite{Li-Poon}
\begin{equation}\begin{split}\label{rel5}
\Lambda_{k}(\textbf{A})&=\set{(\mu_{0}, \mu_{1}, \ldots, \mu_{m})\in\Complex^{m+1} : \exists\,\, P\in\mathcal{P}_{k}\,\, such \,\,\,that\,\,
\right.\\
&\quad\quad\quad\quad\quad\quad\quad\quad\quad\quad\quad\quad\quad\left.PA_{i}P=\mu_{i}P,\, i=0, \ldots, m},
\end{split}\end{equation}
where $\textbf{A}=(A_{0}, A_{1}, \ldots, A_{m})$ is an $(m+1)$-tuple of matrices $A_{i}\in\M_{n}(\Complex)$ for $i=0, \ldots, m$.
Apparently, for $k=1$, $\Lambda_{1}(\mathbf{A})$ is identified with the \emph{joint numerical range}, denoted by
\begin{equation}\label{rel6}
w(\mathbf{A})=\set{(x^{*}A_{0}x, \ldots, x^{*}A_{m}x) : x\in\Complex^{n}, \norm x=1}.
\end{equation}
In the context of quantum information theory, $\Lambda_{k}(\textbf{A})$ is closely
related to a quantum error correcting code, since the latter exists as long as the joint higher rank numerical range
associated with the error operators of a noisy quantum channel is a non empty set.

In section 2, we present some familiar properties of $\Lambda_{k}(L(\lambda))$ and we provide a description of the set
through intersections of numerical ranges of all compressions of the matrix polynomial $L(\lambda)$ to $(n-k+1)$-dimensional
subspaces. This study originates from an analogous expression for matrices, presented and proved in \cite{Aret}.
It also motivates us to investigate the geometry of $\Lambda_{k}(L(\lambda))$ proving  conditions for its
boundedness and elaborating a basic property on the number of its connected components.

In section 3, a connection of the boundary points of $\Lambda_{k}(L(\lambda))$
with respect to the boundary points of $w(L(\lambda))$ is considered. Particularly, introducing the notion of \textit{sharp points}
for $\Lambda_{k}(L(\lambda))$, we show that a sharp point of $w(A\lambda-B)$ with algebraic multiplicity $k$ with respect to the spectrum
$\sigma(A\lambda -B)$ is also a sharp point of
$\Lambda_{j}(A\lambda-B)$, for $j=2, \ldots, k$.
In se\-ction 4, a relationship between $\Lambda_{k}(L(\lambda))$ and $\Lambda_{k}(C_{L}(\lambda))$ is presented, where $C_{L}(\lambda)$ is
the companion polynomial of $L(\lambda)$. Also, we treat a sufficient condition for boundary points of $w(\textbf{A})$ to be boundary points of
$\Lambda_{k}(\textbf{A})$, where $\textbf{A}=(A_{0}, \ldots, A_{m})$ and evenly, we investigate an interplay of
$\Lambda_{k}(L(\lambda))$ and $\Lambda_{k}(\textbf{A})$.

\section{Geometrical Properties}
In the beginning of this section, we present some basic properties as in \cite{Li-Sze} for the higher rank numerical range of a matrix polynomial
$L(\lambda)$.
\begin{prop}
Let $L(\lambda)=\sum_{j=0}^{m}{A_{j}\lambda^{j}}$ be a matrix polynomial, where $A_{m}\neq 0$, then
\begin{description}
  \item[(a)] $\Lambda_{k}(L(\lambda))$ is closed in $\Complex$.
  \item[(b)] For any $\alpha\in\Complex$, $\Lambda_{k}(L(\lambda+\alpha))=\Lambda_{k}(L(\lambda))-\alpha$.
  \item[(c)] If $Q(\lambda)=\sum_{j=0}^{m}{A_{m-j}\lambda^{j}}$ then $\Lambda_{k}(Q(\lambda))\setminus\set{0}=\set{\mu^{-1}:
  \mu \in \Lambda_{k}(L(\lambda))}$.
  \item[(d)] If $A_{i}$, $i=0,\ldots ,m$ have a common totally isotropic subspace $\mathcal{S}=span\set{x_{1}, \ldots, x_{k}}$ with orthonormal
  vectors $x_{j}\in\Complex^{n}$, $j=1, \ldots, k$, i.e. $x_{l}^{*}A_{i}x_{j}=0$ for any $l,j=1, \ldots, k$ and $i=0, \ldots, m$, then
  $\Lambda_{k}(L(\lambda))=\Complex$.
\end{description}
\end{prop}

\begin{prop}\label{pr2}
Let $L(\lambda)=\sum_{j=0}^{m}{A_{j}\lambda^{j}}$ be a matrix polynomial, the following are equivalent:
\begin{description}
  \item[(i)] $\mu\in\Lambda_{k}(L(\lambda))$
  \item[(ii)] there exists $M\in\M_{n,k}(\Complex)$ with $rankM=k$ such that $M^{*}L(\mu)M=0_{k}$
  \item[(iii)] there exists an $L(\mu)$-orthogonal $k$-dimensional subspace $\mathcal{K}$ of $\Complex^{n}$
  \item[(iv)] there exist $\set{u_{i}}_{i=1}^{k}$ orthonormal vectors such that $u_{j}^{*}L(\mu)u_{i}=0$ for every $i,j =1, \ldots, k$
  \item[(v)] there exists a $k$-dimensional subspace $\mathcal{K}$ of $\Complex^{n}$ such that $v^{*}L(\mu)v=0$ for every $v\in\mathcal{K}$
  \item[(vi)] there exists a unitary matrix $U\in\M_{n}(\Complex)$ such that $$U^{*}L(\mu)U=\begin{bmatrix}
                                                                                0_{k} & L_{1}(\mu) \\
                                                                                L_{2}(\mu) & L_{3}(\mu) \\
                                                                              \end{bmatrix},$$
 where $L_{1}(\lambda), L_{2}(\lambda)$ and $L_{3}(\lambda)$ are suitable  matrix polynomials.
\end{description}
\end{prop}
\begin{proof}
The arguments (i)-(vi) are equivalent, since $\mu\in\Lambda_{k}(L(\lambda))$ is equivalent to $0\in\Lambda_{k}(L(\mu))$.
Further, we refer to the Proposition 1.1 in \cite{Choi}.
\end{proof}

\begin{prop}\label{pr3}
Let $L(\lambda)=\sum_{j=1}^{m}{A_{j}\lambda^{j}}$, then
\[
\Lambda_{k}(L(\lambda))\subseteq \Lambda_{k-1}(L(\lambda))\subseteq\ldots\subseteq\Lambda_{1}(L(\lambda)).
\]
\end{prop}
\begin{proof}
For any $j\in\set{2, \ldots, k}$, let $\mu_{0}\in\Lambda_{j}(L(\lambda))$. Then $0\in\Lambda_{j}(L(\mu_{0}))\subseteq\Lambda_{j-1}(L(\mu_{0}))$
and consequently, by $0\in \Lambda_{j-1}(L(\mu_{0}))$, we conclude that $\mu_{0}\in\Lambda_{j-1}(L(\lambda))$.
\end{proof}
\begin{cor}\label{cor4}
Let $L(\lambda)$ be an $n\times n$ matrix polynomial. Then for any $k\leq n$
\[
\Lambda_{k}(\underbrace{L(\lambda)\oplus\ldots\oplus L(\lambda)}_{k})=w(L(\lambda)),
\]
i.e. $\Lambda_{k}(\oplus_{k}L(\lambda))$ is a non-empty set.
\end{cor}
\begin{proof}
Due to Proposition \ref{pr3},
$\mu_{0}\in\Lambda_{k}(\oplus_{k} L(\lambda))\subseteq w(\oplus_{k} L(\lambda))$. Hence
$0\in F(\oplus_{k} L(\mu_{0}))=F(L(\mu_{0}))$, equivalently $\mu_{0}\in w(L(\lambda))$ and then we obtain
$\Lambda_{k}(\oplus_{k} L(\lambda))\subseteq w(L(\lambda))$.
In addition, $\mu_{0}\in w(L(\lambda))\Rightarrow 0\in F(L(\mu_{0}))=\cap_{k}F(L(\mu_{0}))\subseteq\Lambda_{k}(\oplus_{k} L(\mu_{0}))$,
according to a relation in \cite{C-K-Z}. Thus
$w(L(\lambda))\subseteq \Lambda_{k}(\oplus_{k} L(\lambda))$ and the proof is established.
\end{proof}
The following result sketches the higher rank numerical range of a \emph{square matrix} through numerical ranges.
\begin{thm}\label{pr4}
Let $A\in\M_{n}(\Complex)$. Then
\[
\Lambda_{k}(A)=\bigcap_{M}{F(M^{*}AM)},
\]
where $M$ is any $n\times(n-k+1)$ isometry.
\end{thm}
The preceding expression of $\Lambda_{k}(A)$  indicates the ''\emph{convexity of $\Lambda_{k}(A)$}'' in another way, since
the Toeplitz-Hausdorff theorem ensures that each $F(M^{*}AM)$ is convex.
For $k=n$, clearly $\Lambda_{n}(A)=\bigcap_{x\in\Complex^{n}, \norm x=1}F(x^{*}Ax)$ and should be $\Lambda_{n}(A)\neq\emptyset$
\emph{precisely} when $A$ is scalar.

By Theorem \ref{pr4}, we may also describe $\Lambda_{k}(A)$ as intersections of circular discs as in \cite{Bonsall, Bon-Dun-II}, i.e.
\[
\Lambda_{k}(A)=\bigcap_{M}\set{\bigcap_{\gamma\in\Complex}{\mathcal{D}\left(\gamma,\norm{M^{*}AM-\gamma I_{n-k+1}}_{2}\right)}}.
\]

Since $\Lambda_{k}(I\lambda-A)$ is identified with the higher rank numerical range of a matrix $A\in\M_{n}(\Complex)$, Theorem \ref{pr4}
paves also the way for a characterization of $\Lambda_{k}(L(\lambda))$, demonstrated in the next proposition.
\begin{prop}\label{pr5}
Suppose $L(\lambda)=\sum_{j=1}^{m}{A_{j}\lambda^{j}}$, then
\[
\Lambda_{k}(L(\lambda))=\bigcap_{M}{w(M^{*}L(\lambda)M)}=\bigcup_{N}{\Lambda_{k}(N^{*}L(\lambda)N)},
\]
where $M\in\M_{n,n-k+1}(\Complex)$, $N\in\M_{n,k}(\Complex)$ are isometries.
\end{prop}
\begin{proof}
Obviously, by Theorem \ref{pr4}
$$\mu_{0}\in\Lambda_{k}(L(\lambda)) \Leftrightarrow 0\in\Lambda_{k}(L(\mu_{0}))\Leftrightarrow$$
$$0\in\bigcap_{M}{F(M^{*}L(\mu_{0})M)}\Leftrightarrow\mu_{0}\in\bigcap_{M}{w(M^{*}L(\lambda)M)}.$$
Evenly, considering the equation
$\Lambda_{k}(A)=\bigcup_{N}{\Lambda_{k}(N^{*}AN)}$ \cite{Aret}, we have
\[
\mu_{0}\in\Lambda_{k}(L(\lambda))\Leftrightarrow 0\in\Lambda_{k}(L(\mu_{0}))\Leftrightarrow
\]
\[
0\in\bigcup_{N}{\Lambda_{k}(N^{*}L(\mu_{0})N)}\Leftrightarrow\mu_{0}\in\bigcup_{N}{\Lambda_{k}(N^{*}L(\lambda)N)}.
\]
\end{proof}
We should note that Proposition \ref{pr5} provides us  an estimation of the boundary of $\Lambda_{k}(L(\lambda))$ through
the numerical approximation of the nume\-rical range $w(L(\lambda))$. Although the higher rank numerical range
$\Lambda_{k}(I\lambda-A)$ is always connected and convex \cite{Li-Sze,Hugo}, $\Lambda_{k}(L(\lambda))$ need \emph{not satisfy}
these properties, as we will see in the next example.\\
\textbf{Example 1.}
Let
\[
L(\lambda)=3I_{5}\lambda^{3}+\left[\begin{smallmatrix}
                               1 & 2 & 3 & 4 & 5 \\
                               0 & -1 & -2 & -3 & -4 \\
                               i & 2i & 3i & 4i & 5i \\
                               -2 & 1 & 2 & 1 & 2 \\
                               0.3 & 0 & 0 & 0 & 0 \\
                              \end{smallmatrix}\right]\lambda^{2}+\left[\begin{smallmatrix}
                                                     1 & 2 & 0 & 0 & 0 \\
                                                     2 & 3 & 4 & 0 & 0 \\
                                                     0 & 4 & 5 & 6 & 0 \\
                                                     0 & 0 & 6 & 7 & 8 \\
                                                     0 & 0 & 0 & 7 & 8 \\
                                                   \end{smallmatrix}\right]\lambda+\left[
                                                   \begin{smallmatrix}
                                                                        4 & -i & 1 & 0 & -2 \\
                                                                        i & 2i & -6i & 1& 0 \\
                                                                          0 & 1 & 4 & 2 & 0 \\
                                                                        -i & 3i & 0 & 2 & 4 \\
                                                                          3 & 1 & 2 & 4 & 5 \\
\end{smallmatrix}\right]
\]
\begin{center}
    \includegraphics[width=0.35\textwidth]{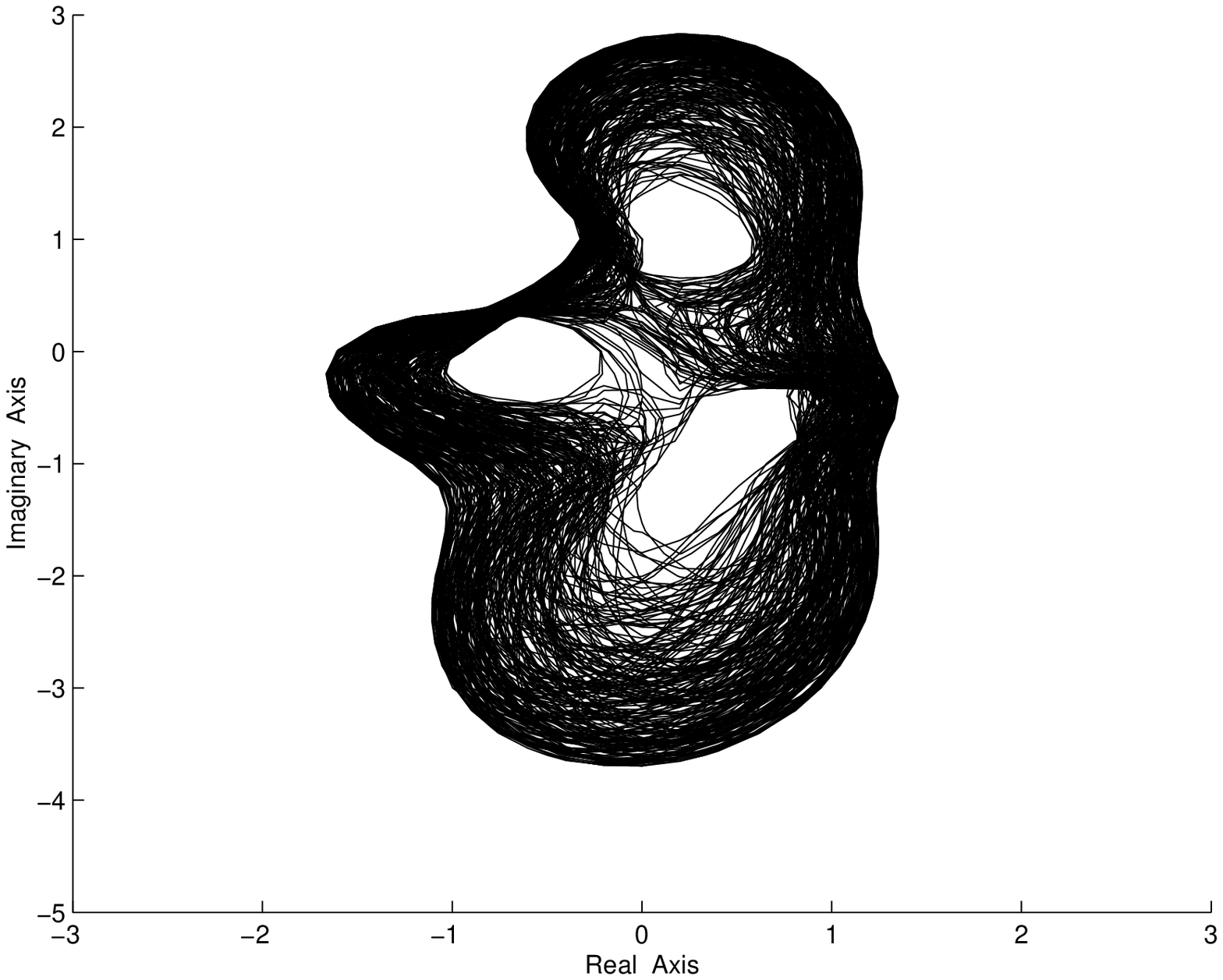}                                              
\end{center}
The intersection of the numerical ranges $w(M^{*}L(\lambda)M)$ by 400 randomly chosen $5\times4$ isometries $M$, approximates the set
$\Lambda_{2}(L(\lambda))$ and it is illustrated by the areas of ''white'' holes inside the figure. Note that all figure constitutes
the nume\-rical range  $w(L(\lambda))$.

Investigating the non emptyness of $\Lambda_{k}(L(\lambda))$, it is noticed  that  the ne\-cessary and sufficient
condition $n\geq 3k-2$ for $\Lambda_{k}(A)\neq\emptyset$ of $A\in\M_{n}$ \cite{Poon-Li-Sze} fails  in general for matrix polynomials,
as shown in the next two results. The first proposition refers to the emptyness  of the set $\Lambda_{k}(A\lambda+B)$, where $A, B$ are
$n\times n$ complex hermitian matrices.
\begin{prop}
Let the $n\times n$ selfadjoint pencil $L(\lambda)=A\lambda+B$ such that $w(M^{*}(A\lambda+B)M)\neq\Complex$ for any $n\times(n-k+1)$ isometry
$M$. If $A$ is a positive semidefinite matrix where the algebraic multiplicity  of the eigenvalue $\mu_{A}=0$ is  greater than $k-1$ and $B$ is
positive (or negative) definite, then $\Lambda_{k}(A\lambda+B)=\emptyset$, for any $k=2, 3, \ldots, n$.
\end{prop}
\begin{proof}
Suppose $B$ is a positive definite matrix. Due to the condition of the multiplicity of $\mu_{A}=0$, the matrices $M^{*}AM$ and $M^{*}BM$ are  positive
semidefinite and positive definite, respectively, for any $n\times(n-k+1)$ isometry $M$.
Moreover, $w(M^{*}(A\lambda+B)M)\neq\Complex$ and  $w(M^{*}(A\lambda+B)M)=(-\infty,-\frac{1}{\nu_{M}}]$, \cite[Th.9]{Psar},
where $\nu_{M}$ is the maximum eigenvalue of $(M^{*}BM)^{-1}M^{*}AM$.  Then, by Proposition \ref{pr5}, we obtain
$$\Lambda_{k}(A\lambda+B)=\bigcap_{M}{w(M^{*}(A\lambda+B)M)}=\bigcap_{M}(-\infty,-\frac{1}{\nu_{M}}]=\Real^{c}=\emptyset.$$
Similarly, if $B$ is a negative definite matrix.
\end{proof}
Moreover, in the next proposition $\Lambda_{k}(L(\lambda))$ appears to be non empty, with $L(\lambda)$ of special form.
\begin{prop}
Let $L(\lambda)=(\lambda-\lambda_{0})^{m}A_{m}$ be an $n\times n$ matrix polynomial,
where $A_{m}\neq 0$ and $0\notin\Lambda_{k}(A_{m})$. Then $\Lambda_{k}(L(\lambda))$ is a singleton, i.e.
$\Lambda_{k}(L(\lambda))=\set{\lambda_{0}}$, $\lambda_{0}\in\Complex$.
\end{prop}
\begin{proof}
Since $0\notin\Lambda_{k}(A_{m})$, by Theorem \ref{pr4}, there exists an $n\times(n-k+1)$ isometry $M_{0}$ such that
$0\notin F(M_{0}^{*}A_{m}M_{0})$ and evenly $w(M_{0}^{*}L(\lambda)M_{0})=w((\lambda-\lambda_{0})^{m}M_{0}^{*}A_{m}M_{0})=\set{\lambda_{0}}$.
Due to the special form of $M^{*}L(\lambda)M$, $\lambda_{0}\in w(M^{*}L(\lambda)M)$
for all $n\times(n-k+1)$ isometries $M$,  whereupon by Proposition \ref{pr5}, we have
\[
\Lambda_{k}(L(\lambda))=\bigcap_{M}{w(M^{*}L(\lambda)M)}=w(M_{0}^{*}L(\lambda)M_{0})=\set{\lambda_{0}}.
\]
\end{proof}
In order to obtain $\Lambda_{k}(L(\lambda))\neq\emptyset$ for any matrix polynomial $L(\lambda)=\sum_{l=0}^{m}{A_{l}\lambda^{l}}$ with $A_{m}\neq 0$,
we are led to the common roots of the $k^{2}>1$ scalar polynomials $b_{ij}(\lambda,Q)=q_{i}^{*}L(\lambda)q_{j}$, $i, j=1, \ldots, k$
for some isometries $Q=\begin{bmatrix}q_{1} & \ldots & q_{k} \\\end{bmatrix}\in\M_{n,k}$.
Adapting the notion of the Sylvester matrix $R_{s}$ appeared in \cite{Mar-Dasc} and the discussion therein to the polynomials
\begin{eqnarray}\label{sylv1}
\nonumber b_{ij}(\lambda,Q)& = & q_{i}^{*}A_{m}q_{j}\lambda^{m}+\ldots +q_{i}^{*}A_{l}q_{j}\lambda^{l}+\ldots+q_{i}^{*}A_{0}q_{j}\\
                           & = & b_{ij}^{(m)}(Q)\lambda^{m}+\ldots+b_{ij}^{(l)}(Q)\lambda^{l}+\ldots+b_{ij}^{(0)}(Q)
\end{eqnarray}
for all $i, j=1, \ldots, k$ and for some $n\times k$ isometry $Q=\begin{bmatrix}q_{1} & \ldots & q_{k} \\ \end{bmatrix}$, we have a condition
for the  polyno\-mials $b_{ij}(\lambda,Q)$ to share polynomial common factors.  Denote by $\sigma\leq m$ to be the largest degree of the $k^{2}$
polyno\-mials $b_{ij}(\lambda,Q)$  and let, as in \eqref{sylv1}
\begin{equation}\label{sylv2}
b_{i_{1},j_{1}}(\lambda,Q)=b_{i_{1},j_{1}}^{(\sigma)}(Q)\lambda^{\sigma}+\ldots+
b_{i_{1},j_{1}}^{(l)}(Q)\lambda^{l}+\ldots+b_{i_{1},j_{1}}^{(0)}(Q),
\end{equation}
for some indices $i_{1}, j_{1}\in\set{1, \ldots, k}$. If $\tau\leq\sigma$ is the largest degree of the remaining polynomials, then the generalized Sylvester matrix is
\begin{equation}\label{sylv3}
R_{s}(Q)=\begin{bmatrix}R_{1}(Q) \\ \vdots \\  R_{k^{2}}(Q) \\  \end{bmatrix},
\end{equation}
where $R_{1}(Q)$ is the stripped $\tau\times (\sigma+\tau)$ matrix
\[
R_{1}(Q)=\begin{bmatrix}
b_{i_{1},j_{1}}^{(\sigma)}(Q) & b_{i_{1},j_{1}}^{(\sigma-1)}(Q) &  & \cdots &  b_{i_{1},j_{1}}^{(0)}(Q) &  & \mathbf{0} \\
  & b_{i_{1},j_{1}}^{(\sigma)}(Q) & b_{i_{1},j_{1}}^{(\sigma-1)}(Q) &  &  &   &  \\
  & \ddots &  & \ddots &  & \ddots &  \\
\mathbf{0} &  & b_{i_{1},j_{1}}^{(\sigma)}(Q) & \cdots & b_{i_{1},j_{1}}^{(\sigma-1)}(Q) & \cdots & b_{i_{1},j_{1}}^{(0)}(Q) \\
\end{bmatrix}
\]
and for $p=2,\ldots, k^{2}$, $R_{p}(Q)$  are the following $\sigma\times(\sigma+\tau)$ matrices
\[
R_{p}(Q)=\begin{bmatrix}
\mathbf{0} &  &  & b_{i_{p},j_{p}}^{(\tau)}(Q) & & \cdot & \cdot & b_{i_{p},j_{p}}^{(0)}(Q) \\
           &  &  b_{i_{p},j_{p}}^{(\tau)}(Q)  &  &  &  &  &  \\
           & \cdot &  & \cdot & & \cdot& \cdot &  \\
b_{i_{p},j_{p}}^{(\tau)}(Q) &  & \cdot & \cdot &  &  &  b_{i_{p},j_{p}}^{(0)}(Q) & \mathbf{0}\\
    \end{bmatrix}
\]
with $i_{p}, j_{p}\in\set{1,\ldots, k}$ and $i_{p}\neq i_{1}$, $j_{p}\neq j_{1}$.
Hence, the degree $\delta(Q)\neq0$ of the greatest common divisor  of $b_{ij}(\lambda,Q)$ $(i,j=1, \ldots, k)$ for some $n\times k$ isometry $Q$
satisfies the relation
\begin{equation}\label{sylv4}
rankR_{s}(Q)=\tau+\sigma-\delta(Q)\leq 2m-\delta(Q)
\end{equation}
and clearly, $\Lambda_{k}(L(\lambda))\neq\emptyset$ if and only if there exists an $n\times k$ isometry $Q$
such that $\textrm{rank}R_{s}(Q)<2m$.

Following, we investigate the boundedness of $\Lambda_{k}(L(\lambda))$ and we state the next helpful lemma.

\begin{lem}\label{lem}
Let $A\in\M_{n}(\Complex)$ with $A\neq 0$. For the function $f:\M_{n,k}(\Complex)\to\M_{k}(\Complex)$ defined by $f(Q)=Q^{*}AQ$, we have
$\textrm{int}(ker f)=\emptyset$.
\end{lem}
\begin{proof}
For $k=1$, let $\textrm{int}(\ker f)\neq\emptyset$ and a vector $x_{0}\in\Complex^{n}\cap \textrm{int}(\ker f)$. Then there exists an  open ball
$\mathcal{B}(x_{0},\varepsilon)\subset\ker f$ with $\varepsilon >0$. For any $y\in\Complex^{n}$ with $y\in\mathcal{B}(0,\varepsilon)$ and real
$t<1$, clearly $ty\in\mathcal{B}(0,t\varepsilon)\subset\mathcal{B}(0,\varepsilon)$ and $x_{0}+ty\in\mathcal{B}(x_{0},\varepsilon)$.
Hence, $$f(x_{0}+y)=f(x_{0}+ty)=0 $$
and consequently we have $(t^{2}-t)y^{*}Ay=0$ for any $y\in\mathcal{B}(0,\varepsilon)$. Therefore, $A=0$, which is a contradiction.

For $k>1$, suppose $Q_{0}\in\M_{n,k}(\Complex)\cap \textrm{int}(\ker f)$ and let the open ball $\mathcal{B}(Q_{0},\varepsilon)\subset\ker f$.
If  an $n\times k$ matrix $Q=\begin{bmatrix}
                        q_{1} & q_{2} & \ldots & q_{k} \\
                      \end{bmatrix}\in\mathcal{B}(Q_{0},\varepsilon)$ and  denote $Q_{0}=\begin{bmatrix}
                        q_{01} & q_{02} & \ldots & q_{0k} \\
                      \end{bmatrix}$, then
\begin{equation}\label{ineq}
\norm{q_{i}-q_{0i}}_{2}=\norm{(Q-Q_{0})e_{i}}_{2}\leq\norm{Q-Q_{0}}_{2}<\varepsilon
\end{equation}
for $i=1, \ldots, k$, where $e_{i}\in\Complex^{n}$ is the $i$-th vector of the standard basis of $\Complex^{n}$ and $\norm\cdot_{2}$ is
the spectral norm. Hence, by
$Q^{*}AQ=Q_{0}^{*}AQ_{0}=0$ we obtain $f(q_{i})=q_{i}^{*}Aq_{i}=0$ and $f(q_{0i})=q_{0i}^{*}Aq_{0i}=0$ $(i=1, \ldots, k)$
and  by \eqref{ineq} we conclude $q_{i}\in\mathcal{B}(q_{0i},\varepsilon)$, i.e.
$\mathcal{B}(q_{0i},\varepsilon)\subset\ker f$. This contradicts the emptyness of  $\textrm{int}(\ker f)$ in the vector case.
\end{proof}
\begin{prop}\label{pr6}
Let $L(\lambda)=A_{m}\lambda^{m}+A_{m-1}\lambda^{m-1}+\ldots +A_{1}\lambda+A_{0}$ be an $n\times n$ matrix polynomial, where $A_{m}\neq 0$.
If $0\notin\Lambda_{k}(A_{m})$, then $\Lambda_{k}(L(\lambda))\neq\emptyset$ is bounded.

Conversely, assume that $rankR_{s}(Q)<2m$, where $R_{s}(Q)$ is the Sylvester matrix  in \eqref{sylv3} of $k^{2}$ scalar polynomials,
elements of matrix $Q^{*}L(\lambda)Q$, for all isometries $Q\in\M_{n,k}$ such that $Q^{*}A_{m}Q=zI_{k}$ $(z\in\Complex\backslash\set{0})$.
If  $\Lambda_{k}(A_{m})\neq\set{0}$ and $\Lambda_{k}(L(\lambda))$ is bounded, then $0\notin\Lambda_{k}(A_{m})$.
\end{prop}
\begin{proof}
Initially, we should remark that we investigate the boundedness of $\Lambda_{k}(L(\lambda))$ taking into account the condition \eqref{sylv4}, so that
it is not empty and all the sets $\Lambda_{1}(L(\lambda))\supseteq\ldots\supseteq\Lambda_{k-1}(L(\lambda))$  are not bounded.
If $0\notin\Lambda_{k}(A_{m})$, then by Theorem \ref{pr4} there exists an $n\times(n-k+1)$ isometry $M_{0}$ such that $0\notin F(M_{0}^{*}A_{m}M_{0})$.
Hence, $w(M_{0}^{*}L(\lambda)M_{0})$ is bounded \cite{Li-Rodman} and by Proposition \ref{pr5}, as
$\Lambda_{k}(L(\lambda))\subseteq w(M_{0}^{*}L(\lambda)M_{0})$, we conclude that $\Lambda_{k}(L(\lambda))$ is bounded.

For the converse, suppose that $0\in\Lambda_{k}(A_{m})\neq\set{0}$ and $\Lambda_{k}(L(\lambda))$ is bounded.
We may find a sequence $\set{z_{\nu}}\subseteq\Lambda_{k}(A_{m})$ such that $\lim_{\nu\to\infty}z_{\nu}=0$ and consequently,
a sequence of $n\times k$ isometries $\set{Q_{\nu}}$ such that $Q_{\nu}^{*}A_{m}Q_{\nu}=z_{\nu}I_{k}\to 0_{k}$.
Due to the compactness of the group of $n\times k$ isometries, there is a subsequence  $\set{Q_{\rho}}$ of
$\set{Q_{\nu}}$ such that $\lim_{\rho\to\infty}Q_{\rho}=Q_{0}$,
with $Q_{0}\in\M_{n,k}$ be an  isometry. Hence, by continuity, $\lim_{\rho\to\infty}Q_{\rho}^{*}A_{m}Q_{\rho}=Q_{0}^{*}A_{m}Q_{0}=0_{k}$ and by Lemma
\ref{lem}, should be  $Q_{\rho}^{*}A_{m}Q_{\rho}=z_{\rho}I_{k}\neq 0$. Note that in \eqref{sylv3}, the Sylvester matrix $R_{s}(Q_{\rho})$
has dimensions $k^{2}m\times 2m$, since in \eqref{sylv2}, $\sigma=\tau=m$ and due to $\textrm{rank}R_{s}(Q_{\rho})<2m$, the equation $Q_{\rho}^{*}L(\lambda)Q_{\rho}=0_{k}$
always guarantees roots.

Moreover, there exists an index $j\neq m$ such that $Q_{0}^{*}A_{j}Q_{0}\neq 0_{k}$ (otherwise $\Lambda_{k}(L(\lambda))\equiv\Complex$)
and evenly, $\norm{Q_{\rho}^{*}A_{j}Q_{\rho}}\geq\varepsilon$ for some fixed $\varepsilon>0$ and  sufficiently large $\rho$.
Hence, the $(m-j)$th elementary symmetric function $\pm\frac{1}{z_{\rho}}Q_{\rho}^{*}A_{j}Q_{\rho}$
of the roots of the matrix polynomial $Q_{\rho}^{*}L(\lambda)Q_{\rho}$ \cite[Th.4.2]{Dennis}
is not bounded, concluding that  $\Lambda_{k}(L(\lambda))$ is not bounded. This contradicts the assumption and the proof is complete.
\end{proof}
Obviously, if $L(\lambda)$ is a monic matrix polynomial, then $\Lambda_{k}(L(\lambda))$ is always bounded.
Following, we present an illustrative example of Proposition \ref{pr6}.\\
\textbf{Example 2.}\\
\textbf{\small{I}.} Let the matrix polynomial
\[
L(\lambda)=\begin{bmatrix}
  1 & 0 & 0 & 0 \\
  0 & i & 0 & 0 \\
  2 & i & 0 & 2 \\
  -i & 0 & -2 & 8 \\
\end{bmatrix}
\lambda^{2}+\begin{bmatrix}
                i & 2 & i & 3 \\
                3 & 0 & 0 & 0 \\
                0 & 4 & 5 & 0 \\
                i & 0 & i & 0 \\
              \end{bmatrix}\lambda+\begin{bmatrix}
                                     1 & 2 & 3 & 4 \\
                                     2 & 3 & 4 & 5 \\
                                     3 & 4 & 5 & 6 \\
                                     5 & 6 & 7 & 8 \\
                                   \end{bmatrix}.
\]
\begin{center}
  \begin{tabular}{cc}
    \includegraphics[width=0.35\textwidth]{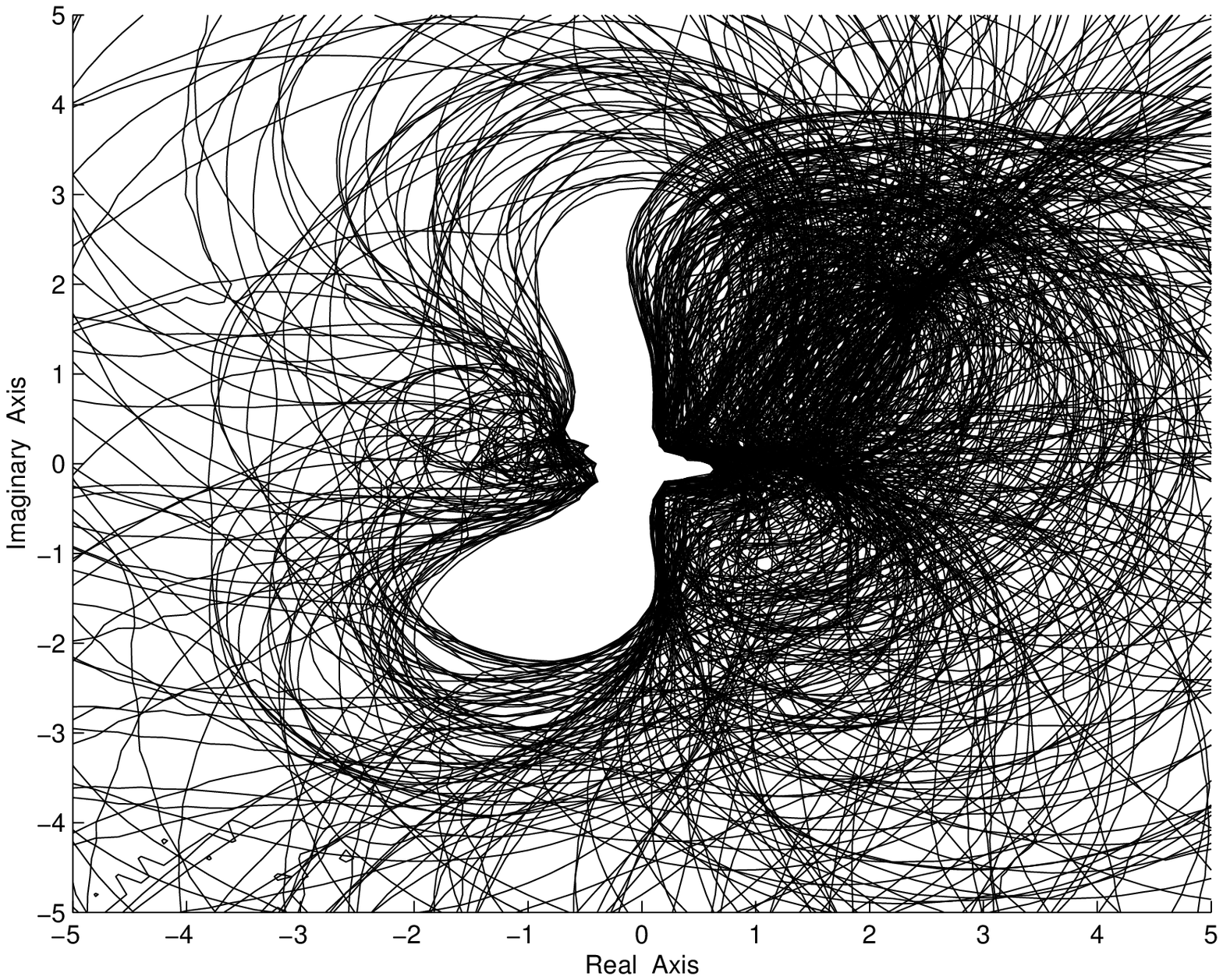}                                       
    \includegraphics[width=0.35\textwidth]{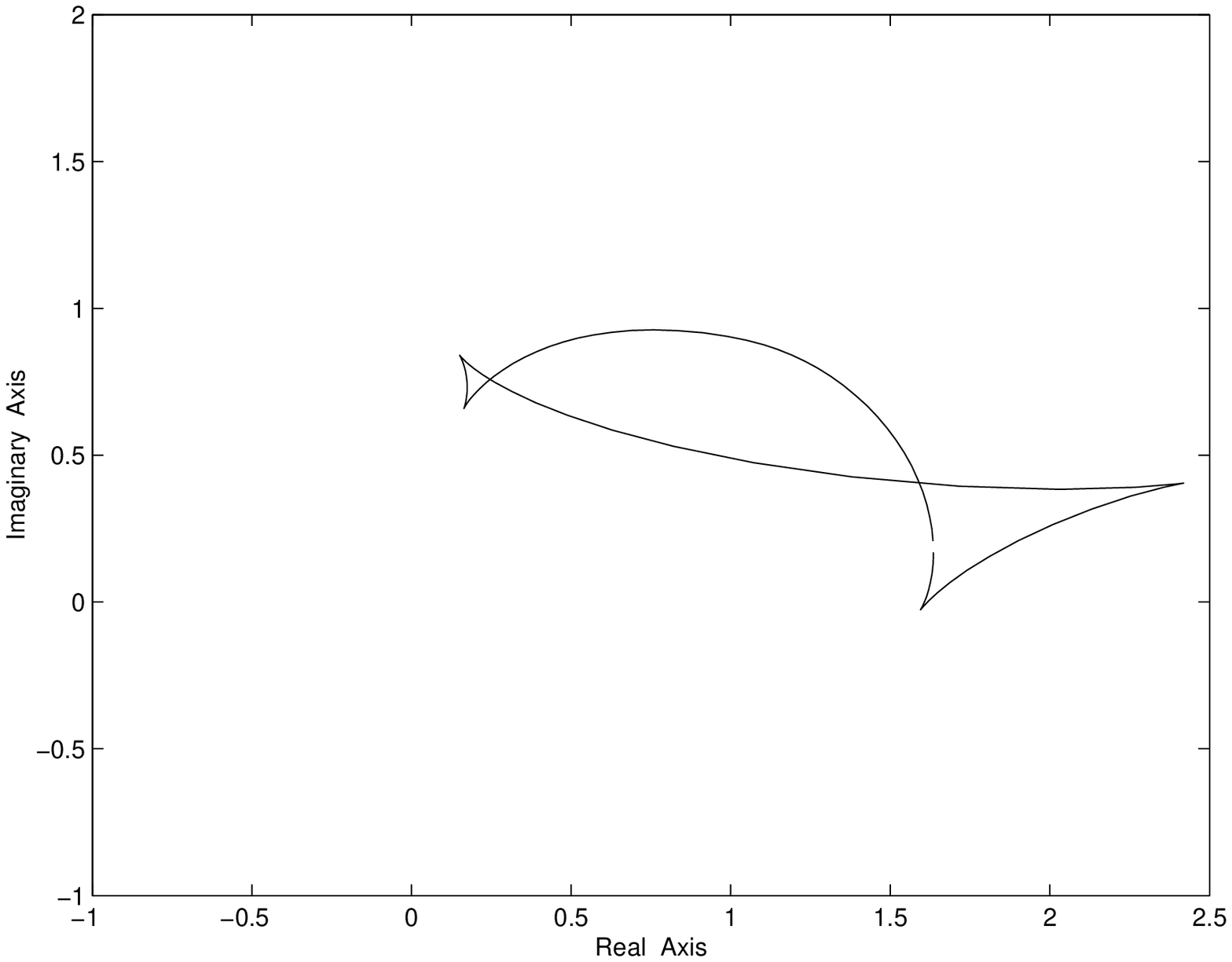}                                       
  \end{tabular}
\end{center}
The uncovered area in the left picture approximates the set $\Lambda_{2}(L(\lambda))$, which is bounded, although
$\Lambda_{1}(L(\lambda))= \Complex$. The boundary of $\Lambda_{2}(A_{2})$ of the leading  coefficient $A_{2}$ is illu\-strated on the right
and we observe that $0\notin\Lambda_{2}(A_{2})$.\\
\textbf{\small{II}.} For the converse, let the $4\times 4$ matrix polynomial $(m=1)$
\[
L(\lambda)=\begin{bmatrix}
         3 & 0 & 0 & 0 \\
         0 & 0 & 0 & 0 \\
         0 & 0 & 0 & 0 \\
         0 & 0 & 0 & 4 \\
       \end{bmatrix}\lambda+\begin{bmatrix}
         0 & 0 & 0 & 0 \\
         0 & 2 & 0 & 0 \\
         0 & 0 & -1 & 0 \\
         0 & 0 & 0 & 0 \\
       \end{bmatrix}=A_{1}\lambda+A_{0}.
\]
Firstly, we observe that $0\in\Lambda_{2}(A_{1})=[0,3]$. On the other hand, $\Lambda_{2}(L(\lambda))$ is equal to the bounded set $\set{0}$.
In fact, if we take the $4\times 3$ isometries
$M_{1}=\left[\begin{smallmatrix}
         1 & 0 & 0 \\
         0 & 1 & 0 \\
         0 & 0 & 0 \\
         0 & 0 & 1 \\
       \end{smallmatrix}\right]$ and $M_{2}=\left[\begin{smallmatrix}
         1 & 0 & 0 \\
         0 & 0 & 0 \\
         0 & 1 & 0 \\
         0 & 0 & 1 \\
       \end{smallmatrix}\right]$, then $M_{1}^{*}A_{1}M_{1}$ and $M_{1}^{*}A_{0}M_{1}$ are both posi\-tive semidefi\-nite matrices and consequently,
\cite[Th.9]{Psar}, $w(M_{1}^{*}L(\lambda)M_{1})=(-\infty,0]$. Similarly, $M_{2}^{*}A_{1}M_{2}$, $M_{2}^{*}A_{0}M_{2}$ are positive and negative
semi\-definite, respectively,  which\, verifies \, $w(M_{2}^{*}L(\lambda)M_{2})=[0,\infty)$. Clearly,
$\Lambda_{2}(L(\lambda))\subseteq w(M_{1}^{*}L(\lambda)M_{1})\cap w(M_{2}^{*}L(\lambda)M_{2})=\set{0}$ and  $0\in\Lambda_{2}(L(\lambda))\neq\emptyset$,
i.e. $\Lambda_{2}(L(\lambda))=\set{0}$.\\
In addition,  for the isometry $Q=\left[\begin{smallmatrix}
         0 &  1/\sqrt{3} \\
         -\sqrt{6}/4 & 1/\sqrt{3}   \\
         \sqrt{6}/4  & 1/\sqrt{3}  \\
         1/2 & 0  \\
       \end{smallmatrix}\right]$ we have $Q^{*}A_{1}Q=I_{2}$ and in \eqref{sylv3} the Sylvester matrix   $R_{s}(Q)=\left[\begin{smallmatrix}
         1 & 3/8  \\
         1 &  1/3 \\
         0 & -3\sqrt{2}/4  \\
         0 & -3\sqrt{2}/4  \\
       \end{smallmatrix}\right]$ has $\textrm{rank}R_{s}(Q)=2$, not less than 2, as it is required.\\
\textbf{\small{III}.} Consider the $4\times 4$ matrix polynomial $L(\lambda)=I_{2}\otimes(B\lambda+I_{2})$, with
$B=\left[\begin{smallmatrix}
         1 & 1 \\
         0 & 0 \\
       \end{smallmatrix}\right]$. Then  $\Lambda_{2}(I_{2}\otimes B)\neq\set{0}$ and additionally, $0\in\Lambda_{2}(I_{2}\otimes B)$.
In this case, for any $4\times 2$ isometry $Q$ such that
$Q^{*}(I_{2}\otimes B)Q=zI_{2}\neq0_{2}$, the Sylvester matrix in \eqref{sylv3}
is $R_{s}(Q)=\left[\begin{smallmatrix}1 & 1/z\\1 & 1/z \\ 0 & 0\\0 & 0\\
\end{smallmatrix}\right]$ with $\textrm{rank}R_{s}(Q)=1<2$. Since, $0\in F(A_{2})$, then $w(L(\lambda))$ as well as
$\Lambda_{2}(L(\lambda)\oplus L(\lambda))$ (Corollary \ref{cor4}) are unbounded. It was
expected by the converse of Proposition \ref{pr6}.\\

Further, we study the connectedness of $\Lambda_{k}(L(\lambda))$, attempting to specify a bound for the number of its connected components.
\begin{prop}\label{con}
Let $L(\lambda)=A_{m}\lambda^{m}+\ldots +A_{1}\lambda+A_{0}$ be an $n\times n$ matrix polynomial, with $A_{m}\neq 0$
and let $\Lambda_{k}(L(\lambda))\neq\emptyset$ have $\rho$ connected components. Moreover,  $rankR_{s}(Q)<2m$,
where $R_{s}(Q)$ is the Sylvester matrix in \eqref{sylv3} of $k^{2}$ polynomials (elements of $Q^{*}L(\lambda)Q$), for any $n\times k$
isometry $Q$ such that $Q^{*}A_{m}Q=\gamma I_{k}$ with $\gamma\in\Lambda_{k}(A_{m})\setminus\set{0}$.

If $\Lambda_{k}(A_{m})\setminus\set{0}$ is connected, then $\rho\leq l\leq m$, where
$l$ is the minimum number of distinct roots of the equation $Q^{*}L(\lambda)Q=0$ for any $n\times k$ isometry $Q$
such that $Q^{*}A_{m}Q=\gamma I_{k}$, with $\gamma\in\Lambda_{k}(A_{m})\setminus\set{0}$.

Otherwise, if $\Lambda_{k}(A_{m})\setminus\set{0}=\mathcal{C}_{1}\cup\mathcal{C}_{2}$, $\mathcal{C}_{1}\cap\mathcal{C}_{2}=\emptyset$
and $\mathcal{C}_{i}$, $i=1,2$ are connected, then $\rho\leq l_{1}+l_{2}\leq 2m$, where  $l_{i}$ is the minimum number of distinct roots
of $Q^{*}L(\lambda)Q=0$ for any $n\times k$ isometry $Q$  that corresponds to points $\gamma\in\mathcal{C}_{i}$,
for $i=1, 2$.
\end{prop}
\begin{proof}
Let $\mathcal{C}_{1}$ be a connected component of $\Lambda_{k}(A_{m})\setminus\set{0}$ and
the $n\times k$ isometries $Q_{0}=\begin{bmatrix}q_{01} & \ldots & q_{0k} \\\end{bmatrix}$,
$Q_{1}=\begin{bmatrix}q_{11} & \ldots & q_{1k} \\ \end{bmatrix}$  correspond to
$Q_{0}^{*}A_{m}Q_{0}=\gamma_{0}I_{k}$ and $Q_{1}^{*}A_{m}Q_{1}=\gamma_{1}I_{k}$,
with $\gamma_{0}, \gamma_{1}\in\mathcal{C}_{1}$. Evenly, we consider that $Q_{0}^{*}L(\lambda)Q_{0}=0_{k}$, $Q_{1}^{*}L(\lambda)Q_{1}=0_{k}$
and in particular, $Q_{0}$ has the pro\-perty that provides the minimum number of distinct roots.
We shall prove that there exists a continuous function of isometries $Q(t):[0,1]\to\M_{n,k}(\Complex)$, with $Q(0)=Q_{0}$, $Q(1)=Q_{1}U$ for some
unitary matrix $U$ such that corresponds to a continuous path $\gamma(t)\in\mathcal{C}_{1}$ joining $\gamma_{0}$ to $\gamma_{1}$.

In case $\gamma_{0}\neq\gamma_{1}$ and the line segment joining $\gamma_{0}$, $\gamma_{1}$ does not contain the origin,
consider the continuous function
\begin{equation}\label{q1}
Q(t)=(\sqrt{1-t^{2}}Q_{0}+tQ_{1}U)C(t,U),\quad t\in[0,1],
\end{equation}
where $U=diag(e^{i\theta_{1}}, \ldots, e^{i\theta_{k}})$, with $\theta_{j}\in[0,2\pi]$, $j=1, \ldots, k$ and
\[
C(t,U)=diag(c_{1}^{-1}(t,\theta_{1}), \ldots, c_{k}^{-1}(t,\theta_{k}))\in\M_{k},
\]
where $c_{j}(t,\theta_{j})=\|\sqrt{1-t^{2}}q_{0j}+te^{i\theta_{j}}q_{1j}\|_{2}$, $j=1, \ldots, k$. Clearly,  $Q(0)=Q_{0}$, $Q(1)=Q_{1}U$ and
$Q^{*}(t)Q(t)=I_{k}$, since the subspaces $\mathcal{K}_{j}=span\set{q_{0j},q_{1j}}$ are pairwise orthogonal for all $j=1, \ldots, k$
\cite[p.318]{Halmos}. Hence, after some manipulations we obtain
\[
Q^{*}(t)A_{m}Q(t)=C(t,U)\left[\gamma(t)I_{k}+t\sqrt{1-t^{2}}(Q_{0}^{*}A_{m}Q_{1}U+U^{*}Q_{1}^{*}A_{m}Q_{0})\right]C(t,U),
\]
where $\gamma(t)=\gamma_{0}+t^{2}(\gamma_{1}-\gamma_{0})$ for $t\in[0,1]$. Moreover, according to the conditions (i)-(iii) in the proof of Theorem 2.2
in \cite{Li-Rodman}, we may have a suitable unitary matrix $U_{0}=diag(e^{i\theta_{01}}, \ldots, e^{i\theta_{0k}})$ such that the matrix function
\[
g(U)=Q_{0}^{*}A_{m}Q_{1}U+U^{*}Q_{1}^{*}A_{m}Q_{0}
\]
satisfies one of the following conditions:
\begin{itemize}
  \item [(i)] $g(U_{0})=0_{k}$,
  \item [(ii)] $g(U_{0})=\xi(\gamma_{1}-\gamma_{0})I_{k}$ for some real $\xi\neq 0$.
\end{itemize}
Then,  $Q^{*}(t)A_{m}Q(t)=\left[\gamma_{0}+(t^{2}+\xi t\sqrt{1-t^{2}})(\gamma_{1}-\gamma_{0})\right]C^{2}(t,U_{0})\neq 0_{k}$ and
for all $j=1, \ldots, k$ the line segments
$h_{j}(t)=\frac{\gamma_{0}+(t^{2}+\xi t\sqrt{1-t^{2}})(\gamma_{1}-\gamma_{0})}{c_{j}^{2}(t,\theta_{0j})}\neq 0$
join the points $\gamma_{0}, \gamma_{1}$  without these necessarily  be endpoints.
Apparently, due to the convexity of $\Lambda_{k}(A_{m})$, we have that the isometries $Q(t)$ generate the line segment $\gamma(t)\in\mathcal{C}_{1}$.

In case the origin belongs to the line segment $[\gamma_{0},\gamma_{1}]$, $(\gamma_{0}\neq\gamma_{1})$,
we may  consider another $\gamma_{2}\in\mathcal{C}_{1}$
such that $\gamma_{2}\neq\gamma_{0}, \gamma_{1}$ and $[\gamma_{0}, \gamma_{2}]\cup[\gamma_{2}, \gamma_{1}]\subseteq\mathcal{C}_{1}$. This is
true because of the convexity of $\Lambda_{k}(A_{m})$ and the fact that the points $\gamma_{0}, \gamma_{1}$ belong to the same connected component.

Finally, if $\gamma_{0}=\gamma_{1}$ and $A_{m}$ is a scalar matrix, then instead of \eqref{q1} consider the conti\-nuous function of $n\times k$
isometries
\[
Q(t)=(\sqrt{1-t^{2}}Q_{0}+tQ_{1})C(t,I_{k}),\quad t\in[0,1].
\]
Otherwise, if $A_{m}$ is not scalar, we refer to  \eqref{q1}.

Thus, we have constructed a continuous function of $n\times k$ isometries $Q(t)$ such that $Q^{*}(t)A_{m}Q(t)=\gamma(t)I_{k}\neq0_{k}$, $t\in[0,1]$
and this asserts that the Sylvester matrix $R_{s}(Q(t))\in\M_{k^{2}m,2m}$ for $t\in[0,1]$, since
$\sigma=\tau=m$ in \eqref{sylv2}. Hence, by the assumption $\textrm{rank}R_{s}(Q(t))<2m$ for all $t\in[0,1]$, we have  that the equation
$Q^{*}(t)L(\lambda(t))Q(t)=0$ has  roots, let $\lambda_{1}(t), \ldots, \lambda_{r}(t)$ $(r\leq m)$.
Due to  the continuity of $Q(t)$, the roots $\lambda_{j}(t):[0,1]\rightarrow\Lambda_{k}(L(\lambda))$ are continuous paths in $\Lambda_{k}(L(\lambda))$,
connecting the\, roots \,of  equations $Q_{0}^{*}L(\lambda)Q_{0}=0_{k}$ and $Q_{1}^{*}L(\lambda)Q_{1}=0_{k}$ and thus the proof is completed.
\end{proof}
\textbf{Example 3.}\\
Let the $4\times 4$ quadratic matrix polynomial
\[
L(\lambda)=\left[\begin{smallmatrix}
             2i & 0 \\
             0 & -2i \\
           \end{smallmatrix}\right]\otimes I_{2}\lambda^{2}+4I_{4}\lambda=\lambda(D\lambda+4I_{4}), \,\,with\,\,\,
D=\left[\begin{smallmatrix}
             2i & 0 \\
             0 & -2i \\
           \end{smallmatrix}\right]\otimes I_{2}.
\]
Obviously, $\Lambda_{2}(L(\lambda))=\set{0}\cup\Lambda_{2}(D\lambda+4I_{4})$ and $0\notin\Lambda_{2}(D\lambda+4I_{4})\neq\emptyset$, that is
$\set{0}$ is an isolated point. We also note that
$\mu_{0}\in\Lambda_{2}(D\lambda+4I_{4})$ if and only if
$\mu_{0}^{-1}\in\Lambda_{2}(\left[\begin{smallmatrix}-i/2 & 0\\0 & i/2\\\end{smallmatrix}\right]\otimes I_{2})\setminus\set{0}$,
therefore $\Lambda_{2}(L(\lambda))$ has three connected components, two on the imaginary axis, the sets $(-\infty,-2]$, $[2,\infty)$ and $\set{0}$.
Moreover, for the $8\times 4$ Sylvester matrix $R_{s}(Q)$ in \eqref{sylv3} we have
$\textrm{rank}R_{s}(Q)=\textrm{rank}\left[\begin{smallmatrix} \lambda_{0} & 0 & 4 & 0\\ 0 & \lambda_{0} & 0 & 4\\
\lambda_{0} & 0 & 4 & 0\\ 0 & \lambda_{0} & 0 & 4\\  \end{smallmatrix}\right]<4$ for all isometries $Q\in\M_{4,2}$ such that
$Q^{*}DQ=\lambda_{0}I_{2}\neq 0_{2}$. Also $\Lambda_{2}(D)\setminus\set{0}$  has
two connected components and Proposition \ref{con} is confirmed.

\section{Sharp points}
In this section, following \cite{Mar-Psar1}, we define  the notion of sharp points. Parti\-cularly, $z_{0}\in\partial\Lambda_{k}(L(\lambda))$
is called to be a \emph{sharp point} if for a connected component
$\Lambda_{k}^{(s)}(L(\lambda))$ of $\Lambda_{k}(L(\lambda))$ there exist a disc $S(z_{0},\varepsilon)$, with $\varepsilon>0$ and two
angles $\theta_{1}< \theta_{2}$, with $\theta_{1}, \theta_{2}\in[0, 2\pi)$, such that
\[
\RE{(e^{i\theta}z_{0})}=\max{\set{\RE z : e^{-i\theta}z\in\Lambda_{k}^{(s)}(L(\lambda))\cap S(z_{0},\varepsilon)}}\quad
\forall \,\,\theta\in(\theta_{1}, \theta_{2}).
\]
The following proposition presents a condition for a boundary point of $w(L(\lambda))$ to be a boundary point of
$\Lambda_{k}(L(\lambda))$, as well. We should remark that the term 'multiplicity' as mentioned below is referred to the
\emph{algebraic multiplicity} of an eigenvalue.
\begin{prop}\label{pr8}
Let the $n\times n$ matrix polynomial $L(\lambda)$. If $\gamma\in\sigma(L(\lambda))\cap\partial w(L(\lambda))$ with
multiplicity $k$, then for $j=2, \ldots, k$
\[
\gamma\in\partial \Lambda_{j}(L(\lambda)).
\]
\end{prop}
\begin{proof}
Clearly, by the assumption, $\gamma$ is seminormal eigenvalue of the matrix polynomial $L(\lambda)$ of multiplicity $k$ \cite[Th.6]{L-P}.
That is, there exists a unitary matrix $U$ such that
\[
U^{*}L(\gamma)U=0_{k}\oplus R(\gamma),
\]
where $R(\lambda)$ is an $(n-k)\times(n-k)$ matrix polynomial and $\gamma\notin\textrm{int}w(R(\lambda))$.
Hence, by Propositions \ref{pr2}(vi) and \ref{pr3}, it is implied that
$\gamma\in\Lambda_{j}(L(\lambda))\subseteq\Lambda_{j-1}(L(\lambda))$ for $j=2, \ldots, k$ and due to
$\gamma\notin \textrm{int} w(L(\lambda))$ $(\equiv \textrm{int}\Lambda_{1}(L(\lambda))$, we obtain
$\gamma\in\partial\Lambda_{j}(L(\lambda))$, for $j=2, \ldots, k$.
\end{proof}
For the pencil  $I\lambda-A$, we obtain the following corollary.
\begin{cor}\label{pr9}
Let $A\in\M_{n}(\Complex)$. If $\gamma\in\partial F(A)$ is eigenvalue of $A$ of multiplicity $k$,
then
\[
\gamma\in\partial \Lambda_{j}(A), \quad j=2, \ldots, k.
\]
\end{cor}
The converse of Corollary \ref{pr9} and consequently of Proposition \ref{pr8} is not true, as it is illustrated in the next example.\\\\
\textbf{Example 4.}\\
Let $A=diag(3+4i, 4-i, -3-2i, -3, -3+3i)$. The outer polygon of the figure is $F(A)$, whereas the inner
shaded polygon  is
$\Lambda_{2}(A)$, which is the intersection of all $\begin{pmatrix}
                                                      5 \\
                                                      4 \\
                                                    \end{pmatrix}$ convex combinations of the eigenvalues
$\lambda_{j_{1}}, \lambda_{j_{2}}, \lambda_{j_{3}}, \lambda_{j_{4}}$ of $A$, with $1\leq j_{1}\leq\ldots\leq j_{4}\leq 5$.
Notice that $\lambda_{0}=-3 \in \partial F(A)\cap\partial\Lambda_{2}(A)$, but it is a simple eigenvalue of matrix $A$.
In addition, $\Lambda_{3}(A)=\emptyset$.

\begin{center}
\includegraphics[width=0.35\textwidth]{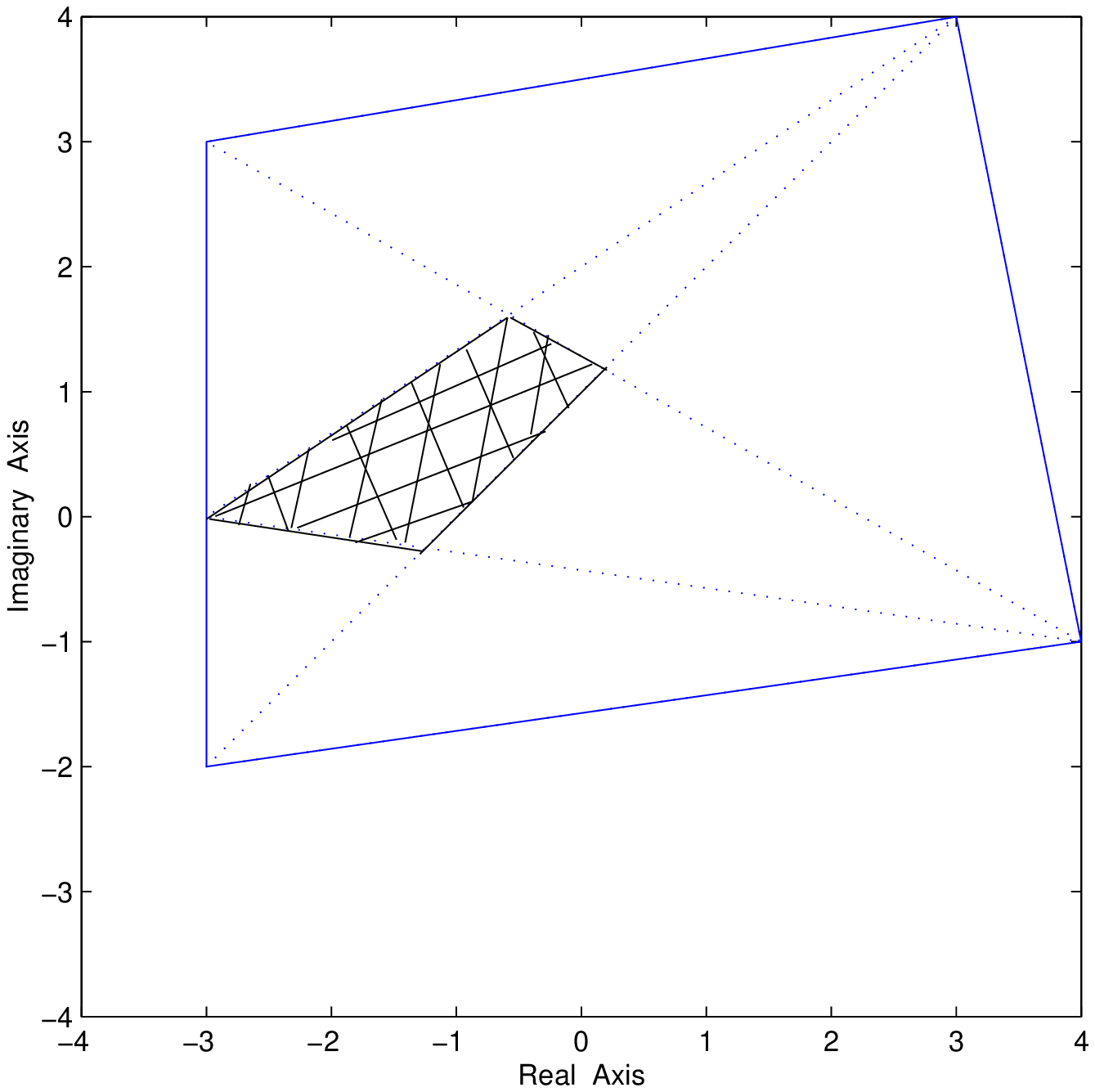}                                       
\end{center}

In view of the definition of sharp points, for a pencil $A\lambda-B$,
we have the next proposition.
\begin{prop}\label{pr10}
Let the  pencil $L(\lambda)=A\lambda-B\in\M_{n}(\Complex)$ and $z_{0}$ be a sharp point of
$w(A\lambda-B)$ of multiplicity $k$ with respect to the spectrum $\sigma(A\lambda-B)$, then $z_{0}$ is also a sharp point of
$\Lambda_{j}(A\lambda-B)$, for $j=2, \ldots, k$.
\end{prop}
\begin{proof}
Since the sharp point $z_{0}$ of $w(A\lambda-B)$ is also an eigenvalue of the pencil $A\lambda-B$ \cite[Th.1.3]{Mar-Psar},
with multiplicity $k$ by hypothesis, we deduce by Proposition \ref{pr8} that $z_{0}\in\partial\Lambda_{j}(A\lambda-B)$,
for $j=2, \ldots, k$. It only suffices to prove that for any disc
$S(z_{0},\varepsilon)$ with $\varepsilon>0$, $z_{0}$ satisfies the equality
\[
\RE(e^{i\theta}z_{0})=\max{\set{\RE z : e^{-i\theta}z\in\Lambda_{j}(A\lambda-B)\cap S(z_{0},\varepsilon)}}
\]
or equivalently, due to Proposition \ref{pr5}
\[
\RE(e^{i\theta}z_{0})=\max{\set{\RE z : z\in \bigcap_{M}{\left(w(e^{i\theta}M^{*}(A\lambda-B)M)\cap S(e^{i\theta}z_{0},\varepsilon)\right)}}}
\]
for every angle $\theta\in(\theta_{1}, \theta_{2})$ with $0\leq\theta_{1}< \theta_{2}<2\pi$.

The inclusion relation $w(M^{*}(A\lambda-B)M)\subseteq w(A\lambda-B)$ for any $n\times(n-j+1)$ isometry $M$, $j=2, \ldots, k$
verifies the inequality
\begin{equation}\label{eq3}
\max_{\bigcap_{M}{\left(w(e^{i\theta}M^{*}(A\lambda-B)M)\cap S(e^{i\theta}z_{0},\varepsilon)\right)}}{\RE z}\leq
\max_{w(e^{i\theta}(A\lambda-B))\cap S(e^{i\theta}z_{0},\varepsilon)}{\RE z}=\RE(e^{i\theta}z_{0})
\end{equation}
for any disc $S(e^{i\theta}z_{0},\varepsilon)$ and every $\theta\in(\theta_{1}, \theta_{2})$.

Moreover, $\ker{(Az_{0}-B)}\cap \IM{(MM^{*})}\neq\emptyset$, since
$\dim{\ker{(Az_{0}-B)}} + \dim{\IM{(MM^{*})}}=k+n-j+1\geq n+1$. Therefore, for an eigenvector $x_{0}\in\Complex^{n}$ of $A\lambda-B$
corresponding to $z_{0}$ there exists
a vector $y_{0}\in\Complex^{n}$ such that $x_{0}=MM^{*}y_{0}$. Obviously, $M^{*}y_{0}\in\Complex^{n-j+1}$ is an eigenvector of
$M^{*}(A\lambda-B)M$ corresponding to $z_{0}$, yielding $z_{0}\in\sigma(M^{*}(A\lambda-B)M)\subseteq w(M^{*}(A\lambda-B)M)$ for
any $n\times(n-j+1)$ isometry M.

Thus, $z_{0}\in \bigcap_{M}{w(M^{*}(A\lambda-B)M)}$, i.e. $\RE z_{0}\in\RE (\bigcap_{M}{w(M^{*}(A\lambda-B)M)})$, whereupon
we confirm the inequality
\begin{equation}\label{eq4}
\RE{(e^{i\theta}z_{0})}\leq\max_{\bigcap_{M}{\left(w(e^{i\theta}M^{*}(A\lambda-B)M)\cap S(e^{i\theta}z_{0},\varepsilon)\right)}}{\RE z}
\end{equation}
for any disc $S(e^{i\theta}z_{0},\varepsilon)$ and every $\theta\in(\theta_{1}, \theta_{2})$.
Therefore, by \eqref{eq3} and \eqref{eq4}
\[
\RE{(e^{i\theta}z_{0})}=\max{\set{\RE z : z\in\bigcap_{M}{\left(w(e^{i\theta}M^{*}(A\lambda-B)M)\cap S(e^{i\theta}z_{0},\varepsilon)\right)}}}
\]
for any disc $S(e^{i\theta}z_{0},\varepsilon)$ and every $\theta\in(\theta_{1}, \theta_{2})$, establishing the assertion.
\end{proof}
Because of the previous results, we obtain an interesting corollary concerning the sharp points of the higher rank numerical range
of a matrix $A\in\M_{n}(\Complex)$.
\begin{cor}\label{pr11}
Let $A\in\M_{n}(\Complex)$ and $z_{0}\in\partial F(A)$ be a sharp point of $F(A)$ of multiplicity $k$ with respect to $\sigma(A)$,
then $z_{0}$ is also a sharp point of $\Lambda_{j}(A)$, for $j=2, \ldots, k$.
\end{cor}

Analogous statement to Proposition \ref{pr10} for the ''sharp points'' of $\Lambda_{j}(L(\lambda))$ we may confirm
taking into consideration Theorem 1.4 in \cite{Mar-Psar}.

\section{Connection between $\Lambda_{k}(L(\lambda))$ and $\Lambda_{k}(\textbf{A})$}

Let the matrix polynomial $L(\lambda)=\sum_{i=0}^{m}{A_{i}\lambda^{i}}$ and the corresponding $mn\times mn$ companion pencil
$$C_{L}(\lambda)=\begin{bmatrix}
                          I_{n} & 0 & 0 & \cdots & 0 \\
                          0 & I_{n} & 0 & \cdots & 0 \\
                          \vdots & & \ddots & \ddots & \vdots \\
                          0 & &  & & 0\\
                          0 & & \cdots & & A_{m} \\
                          \end{bmatrix}\lambda-\begin{bmatrix}
                          0 & I_{n} & 0 & \cdots & 0 \\
                          0 & 0 & I_{n} & \cdots & 0 \\
                          \vdots & & \ddots & \ddots & \vdots \\
                          0 & &  & & I_{n}\\
                          A_{0} & & \cdots & & A_{m-1} \\
                          \end{bmatrix},$$
well known as \emph{linearization} of $L(\lambda)$, since there exist suitable matrix polynomials $E(\lambda)$,
$F(\lambda)$ with constant nonzero determinants such that
\[
\begin{bmatrix}
  L(\lambda) & 0 \\
  0 & I_{n(m-1)} \\
\end{bmatrix}=E(\lambda)C_{L}(\lambda)F(\lambda).
\]

Next, we generalize a corresponding relation in \cite{Mar-Psar1} between the higher rank numerical ranges of
$L(\lambda)$ and $C_{L}(\lambda)$.
\begin{prop}
\quad\quad\quad$\Lambda_{k}(L(\lambda))\cup\set{0}\subseteq\Lambda_{k}(C_{L}(\lambda)).$
\end{prop}
\begin{proof}
By Proposition \ref{pr5} and the relationship $w(L(\lambda))\cup\set{0}\subseteq w(C_{L}(\lambda))$ in \cite{Mar-Psar1}, we have
\begin{equation}\label{s1}
\Lambda_{k}(L(\lambda))\cup\set{0}=\left(\bigcap_{M}w(M^{*}L(\lambda)M)\right)\cup\set{0}\subseteq \bigcap_{M}w(C_{M^{*}LM}(\lambda)),
\end{equation}
where $M\in\M_{n,n-k+1}(\Complex)$, with $M^{*}M=I_{n-k+1}$ and $C_{M^{*}LM}(\lambda)$ is the linearization of the matrix polynomial
$M^{*}L(\lambda)M$. Since,
\begin{eqnarray*}
C_{M^{*}LM}(\lambda)& = & (I_{m}\otimes M)^{*}\begin{bmatrix}
                          \lambda I_{n} & -I_{n} & 0 & \cdots & 0 \\
                          0 & \lambda I_{n} & -I_{n} & \cdots & 0 \\
                          \vdots & & \ddots & \ddots & \vdots \\
                          0 & &  & & -I_{n}\\
                          A_{0} & & \cdots & & A_{m}\lambda+A_{m-1} \\
                          \end{bmatrix}
                          (I_{m}\otimes M)\\
& = & (I_{m}\otimes M)^{*}C_{L}(\lambda)(I_{m}\otimes M),
\end{eqnarray*}
considering the isometry $Q=\begin{bmatrix}I_{m}\otimes M & V \\
              \end{bmatrix}\in\M_{mn,mn-k+1}(\Complex)$, with $Q^{*}Q=I_{mn-k+1}$, we have
\begin{eqnarray}\label {s2}
\nonumber\bigcap_{M}w(C_{M^{*}LM}(\lambda))&=&\bigcap_{M}w((I_{m}\otimes M)^{*}C_{L}(\lambda)(I_{m}\otimes M))\\
&\subseteq&\bigcap_{Q}w(Q^{*}C_{L}(\lambda)Q)\subseteq\bigcap_{X}w(X^{*}C_{L}(\lambda)X)=
\Lambda_{k}(C_{L}(\lambda)),
\end{eqnarray}
where $X\in\M_{mn,mn-k+1}(\Complex)$ with $X^{*}X=I_{mn-k+1}$. Thus by \eqref{s1} and \eqref{s2} the proof is completed.
\end{proof}

Furthermore, $\Lambda_{k}(L(\lambda))$ appears to be associated with the joint higher rank numerical range $\Lambda_{k}(\textbf{A})$
of an $(m+1)$-tuple of $n\times n$ matrices $\textbf{A}=(A_{0}, A_{1}, \ldots, A_{m})$. In fact,
\begin{eqnarray*}
\Lambda_{k}(L(\lambda)) & = & \set{\lambda\in\Complex : PA_{m}P\lambda^{m}+ \ldots +PA_{1}P\lambda +PA_{0}P=0_{n}\,,\,\,\,P\in\mathcal{P}_{k}} \\
                        & \supseteq & \set{\lambda\in\Complex : (\mu_{m}\lambda^{m}+ \ldots +\mu_{1}\lambda +\mu_{0})P=0_{n}\,,\,\,\,
                        (\mu_{0}, \mu_{1}, \ldots, \mu_{m})\in\Lambda_{k}(\textbf{A})} \\
                        & = & \set{\lambda\in\Complex : \mu_{m}\lambda^{m}+ \ldots +\mu_{1}\lambda +\mu_{0}=0\,,\,\,\,
                        (\mu_{0}, \mu_{1}, \ldots, \mu_{m})\in\Lambda_{k}(\textbf{A})} \\
                        & = & \set{\lambda\in\Complex : \seq{(1,\lambda, \ldots, \lambda^{m}),\textbf{u}}=0\,,\,\,\,
                        \textbf{u}=(\mu_{0}, \mu_{1}, \ldots, \mu_{m})\in\Lambda_{k}(\textbf{A})}. \\
\end{eqnarray*}
The above inclusion justifies that $Q^{*}A_{j}Q$ may not be scalar matrices for  $j=0, \ldots, m$ and for
all isometries $Q\in\M_{n,k}(\Complex)$.

The notion of the joint spectrum in \cite{L-P}, leads to an extension of  Proposition \ref{pr8}.
\begin{prop}
Let $\mathbf{A}=(A_{0}, \ldots, A_{m})$ be an $(m+1)$-tuple of $n\times n$ matrices.
If $(\mu_{0}, \ldots, \mu_{m})\in\partial w(\mathbf{A})$ is a normal joint eigenvalue of $\mathbf{A}$ with geometric multiplicity $k$, then
\[
(\mu_{0}, \ldots, \mu_{m})\in\partial \Lambda_{j}(\mathbf{A}),\quad j=2, \ldots, k.
\]
\end{prop}
\begin{proof}
Since $(\mu_{0}, \ldots, \mu_{m})$ is a normal joint eigenvalue with geometric multiplicity $k$ \cite{L-P},
there exists a unitary matrix $U\in\M_{n}(\Complex)$ such that
\[
(U^{*}A_{0}U, \ldots, U^{*}A_{m}U)=(\mu_{0}I_{k}\oplus B_{0}, \ldots, \mu_{m}I_{k}\oplus B_{m}),
\]
where $(B_{0}, \ldots, B_{m})$ is an $(m+1)$-tuple of $(n-k)\times (n-k)$ matrices and $(\mu_{0}, \ldots, \mu_{m})\notin\sigma(B_{0}, \ldots, B_{m})$.
Thus, $(\mu_{0}, \ldots, \mu_{m})\in\Lambda_{k}(\mathbf{A})$. Since the point $(\mu_{0}, \ldots, \mu_{m})\in\partial w(\mathbf{A})$ and
$\Lambda_{j}(\mathbf{A})\subseteq\Lambda_{j-1}(\mathbf{A})$ for every $j=2, \ldots, k$ \cite{Li-Poon},
we establish $(\mu_{0}, \ldots, \mu_{m})\in\partial\Lambda_{j}(\mathbf{A})$ for all $j=2, \ldots, k$.
\end{proof}

Finally, we obtain the following result relative to that in \cite{Tsat-Psar}.

\begin{prop}
Let the matrix polynomial $L(\lambda)=\sum_{i=0}^{m}{A_{i}\lambda^{i}}$. Then
\[
\Lambda_{k}(L(\lambda))\supseteq\set{\lambda\in\Complex : \seq{(1,\lambda, \ldots, \lambda^{m}),\textbf{u}}=0\,,\,\,\,\textbf{u}\in
co\Lambda_{k}(\mathbf{A})},
\]
where $\mathbf{A}=(A_{0}, A_{1}, \ldots, A_{m})$ is the $(m+1)$-tuple of $n\times n$ matrices $A_{i}$.
\end{prop}
\begin{proof}
Let $\Omega=\set{\lambda\in\Complex : \seq{(1,\lambda, \ldots, \lambda^{m}),\mathbf{u}}=0\,,\,\mathbf{u}\in co\Lambda_{k}(\mathbf{A})}$.
To prove the inclusion, suppose $\lambda_{0}\in\Omega$, that is
\begin{equation}\label{eq10}
\seq{(1,\lambda_{0}, \ldots, \lambda_{0}^{m}),\textbf{u}}=0
\end{equation}
for some $\textbf{u}=(u_{0}, u_{1}, \ldots, u_{m})\in co\Lambda_{k}(\textbf{A})$. We have $\Lambda_{k}(\textbf{A})\subseteq\Complex^{m+1}\equiv
\Real^{2m+2}$ and by Caratheodory's theorem in Convex Analysis, there are at most $2m+3$ elements of $\Lambda_{k}(\textbf{A})$ such that
\[
co\Lambda_{k}(\textbf{A})=\set{\sum_{j=1}^{\rho}{\mu_{j}\textbf{u}_{j}} :
\textbf{u}_{j}\in\Lambda_{k}(\textbf{A}),\,\,\mu_{j}\geq 0\,,\,\sum_{j=1}^{\rho}{\mu_{j}}=1,\,with\,\,\rho\leq 2m+3}.
\]
Hence, for $\textbf{u}=(u_{0}, u_{1}, \ldots, u_{m})\in co\Lambda_{k}(\textbf{A})$ there are suitable $\mu_{j}\geq 0$, $\sum_{j}^{\rho}{\mu_{j}}=1$,
$\rho\leq 2m+3$ \, such that
\begin{equation}\label{eq11}
\textbf{u}=\mu_{1}\textbf{u}_{1}+\ldots +\mu_{\rho}\textbf{u}_{\rho}=\mu_{1}\begin{bmatrix}
                       u_{10} \\
                       \vdots \\
                       u_{1m} \\
                     \end{bmatrix} + \ldots + \mu_{\rho}\begin{bmatrix}
                                                       u_{\rho 0} \\
                                                       \vdots \\
                                                       u_{\rho m} \\
                                                     \end{bmatrix},
\end{equation}
where $\textbf{u}_{j}=\begin{bmatrix}
u_{j0} & \ldots & u_{jm}\\
\end{bmatrix}^{T}\in\Lambda_{k}(\textbf{A})$, $j=1, \ldots, \rho$ and by equations \eqref{eq10} and \eqref{eq11}, we obtain:
\begin{eqnarray*}
\seq{(1,\lambda_{0}, \ldots, \lambda_{0}^{m}),\textbf{u}} & = &
\mu_{1}\begin{bmatrix}
  1 & \ldots & \lambda_{0}^{m} \\
\end{bmatrix}\begin{bmatrix}
                       u_{10} \\
                       \vdots \\
                       u_{1m} \\
                     \end{bmatrix} + \ldots + \mu_{\rho}\begin{bmatrix}
  1 & \ldots & \lambda_{0}^{m} \\
\end{bmatrix}\begin{bmatrix}
                        u_{\rho 0} \\
                        \vdots \\
                        u_{\rho m} \\
                         \end{bmatrix}\\
 \end{eqnarray*}
i.e.
\begin{equation}\label{eq12}
\mu_{1}p_{1}(\lambda_{0})+\ldots +\mu_{\rho}p_{\rho}(\lambda_{0})=0,
\end{equation}
where $p_{j}(\lambda)=u_{jm}\lambda^{m}+ \ldots +u_{j1}\lambda+u_{j0}$ for $j=1, \ldots, \rho$.
Evenly, by $\textbf{u}_{j}=(u_{j0}, \ldots, u_{jm})\in\Lambda_{k}(\textbf{A})$, there exist rank-$k$ orthogonal projections
$P_{j}$, $j=1, \ldots, \rho$ such that $P_{j}A_{i}P_{j}=u_{ji}P_{j}$, $i=0, \ldots, m$ and consequently
\begin{eqnarray*}
  p_{j}(\lambda_{0})P_{j} & = & u_{j0}P_{j}+u_{j1}P_{j}\lambda_{0}+\ldots +u_{jm}P_{j}\lambda_{0}^{m}\\
                          & = & P_{j}A_{0}P_{j}+P_{j}A_{1}P_{j}\lambda_{0}+\ldots+P_{j}A_{m}P_{j}\lambda_{0}^{m} \\
                          & = & P_{j}L(\lambda_{0})P_{j},
\end{eqnarray*}
which means that $p_{j}(\lambda_{0})\in\Lambda_{k}(L(\lambda_{0}))$. Due to the convexity of the higher rank numerical range of the matrix
$L(\lambda_{0})$ and the equation \eqref{eq12},
$0\in\Lambda_{k}(L(\lambda_{0}))$, equivalently $\lambda_{0}\in\Lambda_{k}(L(\lambda))$.
\end{proof}
\bibliographystyle{amsplain}

\newpage
\end{document}